\documentclass[10]{amsart}
\usepackage{amsmath}
\usepackage{amssymb}

\usepackage{mathtools}
\usepackage{relsize}    % for \mathlarger to work
\usepackage{graphicx} 

\newtheorem{theorem}{Theorem}[section]
\newtheorem{corollary}{Corollary}[theorem]
\newtheorem{lemma}[theorem]{Lemma}
\newtheorem{Example}{Example}[theorem]
\theoremstyle{definition}

\theoremstyle{remark}

\usepackage[colorlinks]{hyperref}
\usepackage[margin=1.5cm]{geometry}

\begin{document}

\title{Solutions to the matrix Yang-Baxter equation}

\author{Mukherjee, Himadri}
\address{BITS Pilani K. K. Birla Goa Campus, Goa, India}
\email{himadrim@goa.bits-pilani.ac.in}

\author{Ali M, Askar}
\address{BITS Pilani K. K. Birla Goa Campus, Goa, India}
\email{p20190037@goa.bits-pilani.ac.in}

\thanks{AMS Classification 2020. Primary: 15A24, 16T25}

%\date{\today}
\maketitle

\noindent
\begin{abstract}

In this article, we give a few classes of solutions for the Yang-Baxter type matrix equation, $AXA=XAX$. We provide all solutions for the cases when $A$ is equivalent to a Jordan block or has precisely two Jordan blocks. We also have given a few general properties of the solutions of the YB-equation.
\end{abstract}

\section{Introduction}
The Yang-Baxter equation occurs naturally in the context of statistical mechanics. It is also a natural equation in the context of braid groups, where it is a simple identity of two different compositions of braid patterns. The equation in the matrix form has been studied by many authors recently. We will call the following equation the matrix YB-equation, where $A \in M_n(K)$. 

\begin{equation}\label{main_equation}
AXA=XAX
\end{equation}

The Yang-Baxter equations have been studied across the fields of mathematics with numerous fascinating interconnections. From topology to physics to linear algebra and the theory of quantum groups, this equation has generated a feverish rush of quality research among mathematicians and physicists alike. Solving the equations in the setting of linear algebra poses an interesting problem that has a number of applications in other fields. An understanding of the solution in the linear algebra setting will definitely give important insight into the solution of YBE in general.

A number of interesting works have come up since the question posed by Drinfeld in \cite{Drinfeld}, namely the works in the YBE in set-theoretic set up \cite{set_theoretic} exploring the equation and solutions in the category of $Sets$.

The complexity of having $n^2$ non-linear equation in $n^2$ variable to solve, makes it challenging to find complete solutions for arbitrary coefficient matrix. This compelled many researchers to find solutions for particular A, especially having algebraic properties like diagonalizability and commutativity with the solution matrix. Moreover, the non-linearity of the questions paves ways for geometric (or commutative algebraic) techniques such as manifold theory ( Grobner basis) to be applied to get further insight into the space of solutions. We will summarise a few results that deal with YBE.
In \cite{idempotent2},the  authors discuss the solution to YBE when A is an idempotent matrix, by means of diagonalizing A for two distinct eigenvalues. Also, the authors explain the technique to find solutions when $-A = A^2$. The article \cite{Diag2} contains solutions when A is a diagonalizable matrix with spectrum contained in $\{1, \alpha,0\}$, and extended these results for the case of idempotent and rank-one matrices. \\
The authors in \cite{commuting}, found all commuting solutions for A in Jordan canonical form, and corresponding to each Jordan block, authors described solutions in Toeplitz forms.  In \cite{elementary}, authors found complete solutions when A is an elementary matrix of the form $A = I- uv^T$, where $v^Tu \neq 0$, which makes A diagonalizable. In a similar passion, authors in \cite{I-rank2} found solutions when $A = I – PQ^T$, where P and Q are two $n \times 2$ complex matrices of full column rank such that $det(Q^T P) = 0$. In \cite{3eigenvalue}, authors have found solutions for A having eigenvalues $0, \lambda, \mu$ and A diagonalizable by  Jordan decomposition of A. \\
Irrespective of its simple appearance, when the coefficient matrix is nilpotent poses additional challenges to solving it. In \cite{nilpotent}, authors have found all commuting solutions for A being a nilpotent matrix. Later, one of the same authors, along with others, gave an equivalent system for finding commuting solutions to the YBE, when A is nilpotent and has an index 3\cite{nilpotent2}. \\
In \cite{rank1} author finds complete solutions for YBE when A is rank 1 matrix using the spectral decomposition $A=pq^T$, where p and q are non-zero complex vectors. In \cite{rank2.1} the authors have studied the solution when $A=PQ^T$, where P and Q are $n\times 2$ vectors, with an assumption $QP^T$ non-singular. In an extension study \cite{rank2.2} the same authors have studied the solutions for $QP^T$ being singular with the algebraic multiplicity of eigenvalue 0 being more than $n-2$. In \cite{A4=A} authors have studied non-commuting solutions when A satisfies $A^4 = A$, by exploring various possibilities for its minimal polynomials. \\
 In \cite{YBE2}, authors have given a new horizon for the discussion of the solution to the YBE, by tracing the path-connected subsets of the solution.  Also, the authors explain techniques to generate infinitely many solutions  from existing solutions using generalized inverse. 
 
 \textbf{In \cite{manifold}, author has given manifold structure for the solution set of matrix YBE when the coefficient matrix is  having rank 1.
  In \cite{commuting-nonsingular}, authors have found all commuting solutions of YBE when the coefficient matrix is non-singular.}

In our paper, the results are organized in the manner that there are global results, i.e the results without any assumption on the coefficient matrix or the type of the solutions and special results which come by with assumptions on the coefficient matrix and/or the type of solutions. Before every result, we have mentioned whether it is a global or a local result to facilitate the reader. There are a few minor overlaps in the results, especially with the newest articles, but each time we have presented a new way of looking at it or have extended the result. Which is why we have not removed those results from our paper but have mentioned whenever such overlaps occurred. 

\section{Main Results}
For the coefficient matrix $A$, having a single Jordan block with non-zero eigenvalue, we have the following theorem completely characterizing the solution of the Yang-Baxter equation, $AXA = XAX$.
\begin{theorem}
If $A=\lambda I+B$ is a Jordan block, then for the YBE, $AXA = XAX$, the following are true.
\begin{enumerate}
    \item If $\lambda \neq 0$ and $det(X) \neq 0$ then $X\simeq A$.
    \item If $\lambda \neq 0$ and $det(X)=0$ then $X=0$.
    \item If $\lambda =0$, then the YBE do not have any invertible solution X.  
\end{enumerate}
\end{theorem}

For the case where $A$ has two Jordan blocks of the same size, and the determinant of $A$ is not zero, we have the following theorem completely characterizing the solutions of the YBE, $AXA = XAX$. Let $Sol_A$ denote the set of all solutions of the YBE.
\begin{theorem}
Let $\mathcal{A}$ be a matrix having only two eigenvalues $\lambda_1, \lambda_2 \neq 0$, with the same algebraic multiplicity, and geometric multiplicity 1. Further, let us assume 0 is an eigenvalue of $X$. Let $U \in GL_n(K)$ be defined as, the columns of $U$ are generalized eigenvectors of $\mathcal{A}$,  such that $ U\mathcal{A}U^{-1}$ is Jordan-canonical form  of $\mathcal{A}$. Then one of the following is true. 
\begin{enumerate}
    \item $X=0$
    \item $X = U^{-1} \begin{pmatrix}
    0 & ZS^{-1}A \\
    0 & Y_2
    \end{pmatrix}U$  or $U^{-1} \begin{pmatrix}
    0 & 0 \\
    0 & Y_2
    \end{pmatrix} U$, where $Y_2 \in Sol_{A}$, $Z \in K[A]$, $A$ is the Jordan block corresponding to eigenvalue $\lambda= \lambda_1 = \lambda_2$, and for some $S \in GL_n(K)$
    \item $X = U^{-1} \begin{pmatrix}
    Y_1 & 0 \\
    ZS^{-1}A & 0
    \end{pmatrix} U$ or $U^{-1} \begin{pmatrix}
    Y_1 & 0 \\
   0 & 0
    \end{pmatrix} U$, where $Y_1 \in Sol_{A}$, $Z \in K[A]$, $A$ is the Jordan block corresponding to eigenvalue $\lambda= \lambda_1 = \lambda_2$, and for some $S \in GL_n(K)$
\end{enumerate}
\end{theorem}

\section{preliminaries}
%Let $A$ be a single Jordan block namely $J_{\lambda}=\lambda I + B$ where %$B$ is the nilpotent matrix associated to the Jordan block. 
Let us recall the following notations, which will be used throughout the discussion.
\begin{itemize}
    \item Let $K$ be a field. For a coefficient matrix $A \in M_n(K)$, the set of all solutions to the Yang-Baxter equation $AXA = XAX$ is denoted by $Sol_A := \{X \in M_n(K)| AXA = XAX\}$.
    \item For a matrix $M \in M_n(K)$, $\sigma(M)$ denotes the spectrum of $M$.
    \item  Let $V$ be a vector space, then we say $A \in M_n(K)$ preserves $V$, if $A(V) \subseteq V$.
    \item For a $M \in M_n(K)$, $min_M$ represent the minimal polynomial of $M$.
    \item  $K[A]$ represent the polynomial ring on $A \in M_n(K)$, over $K$.
    \item Throughout the discussion, $B$ represent the nilpotent block $\begin{pmatrix}
        0&1&0&\dots&0\\
        0&0&1&\dots&0\\
        .\\
        .\\
        .\\
        0&0&0&\dots&1\\
        0&0&0&\dots&0
    \end{pmatrix},$ with appropriate dimension. 
\end{itemize}    
\begin{theorem}[\cite{sylvester-original}]\label{sylvester}
     Given matrices $ A\in \mathbb {C} ^{n\times n}$ and $ B\in \mathbb {C} ^{m\times m}$, the Sylvester equation $AX+XB=C$ has a unique solution $ X\in \mathbb {C} ^{n\times m}$ for any $ C\in \mathbb {C} ^{n\times m}$, if and only if A and -B do not share any eigenvalue.
\end{theorem}

\begin{lemma} For any $g \in GL_n(K)$, $g Sol_{A} g^{-1} = Sol_{gAg^{-1}}$.
\end{lemma}
\begin{proof}\label{lemma1}
Let $B= gAg^{-1}$, and $x \in g Sol_{A} g^{-1}$. Then $x= gXg^{-1}$, for some $X \in Sol_A$.
Now for $y= BxB = gAg^{-1} x gAg^{-1}$, that is $g^{-1} y g = A(g^{-1} x g)A
= AXA =XAX = g^{-1}xgAg^{-1}xg = g^{-1} xBx g$
that is, $y = xBx$.
which implies $BxB=xBx$
\end{proof}
As an immediate consequence of the above Lemma \ref{lemma1}, we have the following.
\begin{corollary}
Without loss of generality, we can assume A to be in the Jordan-canonical form.
\end{corollary}
\begin{corollary}
Let $X \in Sol_A$ , then for any $g \in G_A = \{g \in GL_n(K)/ gAg^{-1} =A\}$, $gXg^{-1}$ is a solution to the YBE, $AXA = XAX$.
\end{corollary}
\begin{proof}
Let $X \in Sol_A$. For a $g \in G_A$, we have $gSol_A g^{-1}\ =\ Sol_{gAg^{-1}} = Sol_A$.
\end{proof}
\begin{lemma} \label{first} Le $X \in Sol_A$ for the YBE with coefficient matrix $A$, and $\lambda$ be an eigenvalue of $A$ with eigenvector $v$, then either $\lambda$ is an eigenvalue of $X$ or $AXv = 0$. 
\end{lemma}
\begin{proof}
We have, $Av=\lambda v$, thus $XAv=\lambda Xv$ or $AXAv=\lambda AXv$. By utilizing the YBE, we get $XAXv=\lambda AXv$. Thus, if $AXv \neq 0$, then $\lambda$ is an eigenvalue of $X$.
 Further, by the symmetry of the YBE, we can see, if $\lambda$ is an eigenvalue of $X$ with eigenvector $v$, then either $\lambda$ is an eigenvalue of $A$ or $XAv = 0$. 
\end{proof}

\begin{corollary}\label{cor.inclusion.of.spectrum}
Let coefficient matrix A be invertible and, $X \in Sol_A$, then $\sigma(X) \subseteq \sigma(A)\cup\{0\}$.
\end{corollary}
\begin{proof}
From the above Lemma \ref{first}, we know that if $\lambda$ is an eigenvalue of $X$, then either it is an eigenvalue of $A$ or $XAv=0$, for $v$, an eigenvector of $X$, with the eigenvalue $\lambda$. Now, since $A$ is invertible, we get $Av \neq 0$. Thus, $XAv=0$, implies $Av$ is an eigenvector of $X$ with eigenvalue $0$. 
\end{proof}

\begin{lemma}\label{preserver}
If $det(A) \neq 0$, then $A$ preserves the kernel of $X$.
\end{lemma}
\begin{proof}
Let $v \in Ker(X)$, we have 
$Xv=0
\iff AXv=0  \implies XAXv=0
\iff AXAv=0
\iff XAv=0 $ $\Rightarrow Av \in Ker(X)$.
This gives, $A(Ker (X)) \subseteq Ker(X)$. Since $A$ is invertible, we get $dim (A(Ker (X))) = dim (Ker (X))$, gives, $A(Ker (X)) = Ker (X)$. 
\end{proof}

\begin{lemma}\label{main_lemma}
Let the coefficient matrix $A$ be invertible with $\lambda \in \sigma(A)$ and $E_{\lambda}$, the eigenspace of $A$ corresponding to $\lambda$, be of one dimension. If $Ker(X) = E_{\lambda}$, then $\lambda \notin \sigma(X)$.
\end{lemma}
\begin{proof}
If possible, let us assume $\lambda \in \sigma(X)$. Let $w$ be an eigenvector for $\lambda$, thus we have $Xw = \lambda w$. Now since $\lambda \in \sigma(A)$, let us choose an eigenvector for it, say $v$, we have $Av=\lambda v$. Since we have $E_{\lambda} = Ker(X)$ we obtain $Xv=0$. We also have $XAXw=\lambda XAw$, or $AXAw = \lambda XA w$, using the YBE. Which gives us $XAw \in E_{\lambda}$, now since the eigenspace is one dimensional (generated by $v$ as a basis), we have $XAw=\mu v$ for some $\mu$ (could also be zero). This gives us $X^2AXw=0$ or $X(XAX)w=0$. Or $XAX (Aw)=0$, which gives us $AXA^2w=0$, thus $XA^2w=0$ (as $A$ is invertible). So we have $A^2w \in ker(X)$. Now since $ker(X)=E_{\lambda}$ we have $A^2w=l v$ for some $l \in K$. So we have:
\begin{equation}
    \begin{aligned}
A^2w &=(lv=l/\lambda) Av\\
Aw &=(l/\lambda) v\\
w &= (l/\lambda) A^{-1}v\\
w &=(l/\lambda^2) v
 \end{aligned}
\end{equation}
Thus $w \in E_{\lambda}$ but then $Xw=0=\lambda w$ giving us $w=0$ a contradiction.

\end{proof}
\begin{corollary}\label{fact_two}
Let the coefficient matrix $A$ be invertible, and if  $\ \forall\ \lambda \in \sigma(A),\ dim(E_{\lambda})\ = 1$, and $X (E_{ \lambda }) \neq 0$, then $\sigma(A)\subseteq \sigma(X)$.
\end{corollary}
\begin{proof}
By the Lemma \ref{first}, if $X (E_{\lambda}) \neq 0$, then $\lambda \in \sigma(X)$. If this is true for all $\lambda \in \sigma(A)$, then we have $\sigma(A) \subseteq \sigma(X)$.
\end{proof}
\begin{lemma}\label{main_lemma_three}
Let A be an invertible coefficient matrix, $\lambda \in \sigma(A)$ with geometric multiplicity 1 and $\lambda \notin \sigma(X)$. If $P_{\lambda}$ is the generalized eigenspace of $A$ corresponding to $\lambda$ , then $X(P_{\lambda})=0$.
\end{lemma}
\begin{proof}
Since $\lambda \notin \sigma(X)$, we know that $X$ annihilates the eigenspace of $\lambda$ so $E_{\lambda} \subset Ker(X)$. Let $\{v_{1},\dots, v_{r}\}$ be the canonical basis for the generalised eigen space $P_{\lambda}$, such that $Av_{1}= \lambda v_{1}$, and $Av_{i}= v_{i-1}+ \lambda v_{i}$, for $i=2,\dots,r$. Then we have, $Xv_{1}=0$. Now, $XAv_{2}= Xv_{1}+ \lambda Xv_{2}$, or $AXAv_{2}= \lambda AXv_{2}$. Then by YBE, $XAXv_{2}= \lambda AXv_{2}$. This gives, if $AXv_{2} \neq 0$, then $\lambda \in \sigma(X)$, which is not true. Hence $Xv_{2}=0$. Continuing this process, by induction we get $Xv_{i}=0$, for $i=1,2,\dots,r$. Which implies $XP_{\lambda} =0$.  
\end{proof}
\begin{corollary}\label{fact_one}
Let $A$ be an invertible coefficient matrix for the YBE, and for a $ \lambda \in \sigma(A)$, the Jordan block corresponding to $\lambda$, $J_\lambda$ has size strictly greater than 1 and has geometric multiplicity 1, then for any $X \in Sol_A$, such that $\lambda \notin \sigma(X)$, $Ker (X) \neq E_\lambda$.
\end{corollary}
\begin{proof}
If, for a $\lambda \in \sigma(A)$, $\lambda \notin \sigma(X)$, with $dim (E_\lambda) = 1$, then, by Lemma \ref{main_lemma_three}, we have $X$ annihilate any cyclic subspace of generalised eigenspace of $A$, corresponding to $\lambda$. Let $P_\lambda$ be the maximum cyclic subspace of such generalised eigenspace. Then, $P_\lambda \subseteq Ker(X)$.  Since, $dim (P_\lambda)$ is same as Jordan block corresponding to $\lambda$, which is strictly greater than one, and $E_\lambda$ has dimension one, we have $Ker (X) \neq E_\lambda$.    

\end{proof}

\begin{lemma}\label{both_invertible}
Let A be an invertible coefficient matrix, and $X \in Sol_A$ is also invertible, then $A \simeq X$. In particular, they have the same Jordan form.
\end{lemma}
\begin{proof}
The result follows from the fact, \[AXA=XAX \iff X^{-1}AX=AXA^{-1}\]
\end{proof}
\begin{lemma}\label{second}
If the coefficient matrix $A$ is a Jordan block, $\lambda I +B$, such that $\lambda \neq 0$, and if a solution of the YB equation $X$ is not the zero matrix, then $\sigma(X) = \{\lambda\}$.
\end{lemma}
\begin{proof}
We have $\sigma(X) \subseteq \sigma(A) \cup \{0\}$, and  $\sigma(A) = \{\lambda\}$. Also, we have $A(Ker(X)) \subset Ker(X)$, by Corollary \ref{cor.inclusion.of.spectrum}. So, $Ker(X)$ is an invariant subspace of $A$. i.e., $Ker(X) = E_{\lambda}\mbox{ or } K^n$, as $\lambda$ is the only eigenvalue of $A$, $P_{\lambda} = K^n$. Now, if $Ker(X) = K^n$, then $X = 0$. If $Ker(X) = E_{\lambda}$, then by Lemma \ref{main_lemma}, we have $\lambda \notin \sigma(X)$; which gives $\sigma(X) \subset \{0\}$. i.e., $X = 0$. As a result, if $X \neq 0$, then $\sigma(X) = \{\lambda\}$.
\end{proof}
\begin{lemma}\label{powerlemma}
$AXA=XAX$ if and only if $AXA^n =X^nAX$ for all $n \in \mathbb{N}$. Similarly, $AXA=XAX$ if and only if  $A^nXA=XAX^n, \ \forall\  n \in \mathbb{N}$.
\end{lemma}
\begin{proof}
\[AXA=XAX \]
\[ \Rightarrow AXA^2=XAXA=X^2AX. \] 
Then by induction, we have, $\forall\ n \in \mathbb{N}$, $AXA^n=X^nAX$. Conversely, if $AXA^n=X^nAX$, is true for any $n \in \mathbb{N}$, then in particular, it is true for $n = 1$. i.e., $AXA = XAX$. Similar way, we can prove,  $AXA=XAX \ \iff\ A^nXA=XAX^n, \ \forall\  n \in \mathbb{N}$ 
\end{proof}
\begin{lemma}
If $XAX = AXA$, and $X$ and A commute, then $(I-e^{t(A-X)})XA = 0$, for any $t \in \mathbb{R}$.
\end{lemma}
\begin{proof}
If $XAX = AXA$, then as a consequence of above lemma, we have $XAe^{tX} = e^{tA}XA$ for any $t \in \mathbb{R}$. \\
Which gives 
\begin{align*}
    XAe^{tX}X &= e^{tA}XAX\\
    XAXe^{tX} &= e^{tA}XAX\\
    AXAe^{tX} &= e^{tA}AXA\\
              &= Ae^{tA}XA\\
\end{align*}
This gives $XA = e^{t(A-X)}XA$, implies $(I-e^{t(A-X)})XA = 0$.
\end{proof}
\begin{lemma}
    Let $X \in Sol_A$, be a commuting solution, such that $A - X$ is invertible, then $AX = 0$.
\end{lemma}
\begin{proof}
    If $X$ is a commuting solution in $Sol_A$, then we have, $A^2X = AX^2$. By taking $Y= AX$, we have, $AY - YX = 0$. By Sylvester's Theorem \ref{sylvester}, $Y = 0$ if and only if, $\sigma(A) \cap \sigma(X) = \varnothing$, i.e., $A-X$ is invertible.
\end{proof}
\begin{theorem}
Let $\phi_A(x)$ be the characteristic polynomial of A (respectively $\phi_X(x)$ be the characteristic polynomial of X ). Then the following is true,
\begin{itemize}
    \item[i]$XA\phi_A(X) = 0$
    \item[ii] $\phi_A(X)AX =0$
\end{itemize}
\end{theorem}

\begin{proof}
Let\[ \phi_A(x)= a_0+a_1x+\dots+a_nx^n. \] Then we have,

    \begin{align*}
 XA\phi_A(X) &= a_0XA+a_1XAX+a_2XAX^2+\dots+ a_nXAX^n \\
 &= a_0XA+a_1AXA+a_2A^2XA+\dots+ a_nA^nXA \\
 &= (a_0I+a_1A+a_2A^2+\dots+ a_nA^n)XA \\
 &= \phi_A(A)XA=0 \\
 \end{align*}
Similarly, we have $\phi_A(X)AX=0$.
\end{proof}
\begin{corollary}
If $A^n=0$, then $XAX^m=X^mAX=0$, for every $ m \geq n$.
\end{corollary}
\begin{lemma}
Let $A, X$ be two square matrices of the same dimension. $\phi_A(X)$ is invertible if and only if $\sigma(A) \cap \sigma(X)= \varnothing$.
\end{lemma}
\begin{proof}
Let $\phi_A(X)= (X-\lambda_1I) \dots (X-\lambda_nI)$. Then $\phi_A(X)$ is invertible if and only if for any $i$,  $\lambda_i \notin \sigma(X)$.
\end{proof}
\begin{theorem}\label{zero-solution}
Let $AXA=XAX$, if $\sigma(A) \cap \sigma(X)= \varnothing$, then either
\begin{itemize}
    \item[i] $A=0$ and $X$ is invertible\\
    or
    \item[ii] $A$ is invertible and $X=0$.
\end{itemize}
\end{theorem}
\begin{proof}
As $\sigma(A) \cap \sigma(X)= \varnothing$, only atmost one of $\sigma(A), \sigma(X)$ can contain 0. If none of them contains 0, then both $A, X$ are invertible and hence $X\simeq A$, which implies $\sigma(A) = \sigma(X)$, which is not possible.
Therefore, one must contain 0. 
If $0 \in \sigma(A)$, then  $0 \notin \sigma(X)$ and X is invertible. Also, since $\sigma(A) \cap \sigma(X)= \varnothing$, we have $min_A, min_X$ are co-prime. As a result, $\phi_A(X) $ is invertible.
Then, \[XA\phi_A(X)= \phi_A(X)AX = 0 \\
\Longrightarrow XA= 0.\] Since X is invertible, A must be 0.
Similarly, if $0 \in \sigma(X)$, then $X=0$ and A is invertible.  
\end{proof}

\begin{theorem}\label{main_thm_2}
If $A=\lambda I+B$ is a Jordan block, then for the YBE, $AXA = XAX$, the following are true.
\begin{enumerate}
    \item If $\lambda \neq 0$ and $det(X) \neq 0$ then $X\simeq A$.
    \item If $\lambda \neq 0$ and $det(X)=0$ then $X=0$.
    \item If $\lambda =0$, then the YBE do not have any invertible solution X.  
\end{enumerate}
\end{theorem}
\begin{proof}
    The first case follows from Lemma \ref{both_invertible}. Now, by Lemma \ref{second}, if $A$ is invertible and $X \neq 0$, the $\sigma(X) = \{\lambda\}$, i.e $X$ is invertible. Thus, if $X$ is singular, $X$ must be $0$ matrix. Let, $A = B$, i.e., $\lambda = 0$, and assume if possible, $X$ is invertible. Then, $\sigma(A) \cap \sigma(X) = \varnothing$. Thus, by Theorem \ref{zero-solution}, $A$ must be $0$, which is not true. Hence, $X$ must be singular in this case.    
\end{proof}
\begin{corollary}
Let the coefficient matrix $A= \begin{pmatrix}
A_1&0&\dots&0\\
0&A_2&\dots&0\\
.\\
.\\
.\\
0&0&\dots&A_m
\end{pmatrix}$, where $A_i$'s are square matrices, then $X= \begin{pmatrix}
X_1&0&\dots&0\\
0&X_2&\dots&0\\
.\\
.\\
.\\
0&0&\dots&X_m
\end{pmatrix} \in Sol_A$, where $X_i \in Sol_{A_i}$. 
\end{corollary}
\subsection{Some special solutions }
\begin{Example}\label{ex1}
When $A=\lambda I + B$, for some $\lambda \neq 0 $, is a $2\times 2$ Jordan block, non-trivial solution $X$ has one of the following form: \\
$\begin{pmatrix}
\lambda&1\\
0&\lambda
\end{pmatrix}, 
\begin{pmatrix}
\lambda+\lambda\sqrt{a}& a\\
-\lambda^2 & \lambda-\lambda \sqrt{a}
\end{pmatrix} or 
\begin{pmatrix}
\lambda-\lambda\sqrt{a}& a\\
-\lambda^2 & \lambda+\lambda \sqrt{a}
\end{pmatrix}$ where a $\in K$\\
\end{Example}
\begin{proof}
Let $X= \begin{pmatrix}
x_{1}&x_{2}\\
x_{3}&x_{4}
\end{pmatrix} $, then for $A= \lambda I_{2}+B$ \\
\begin{equation}
    \begin{aligned}
AXA &= \begin{pmatrix}
\lambda&1\\
0&\lambda
\end{pmatrix} 
\begin{pmatrix}
x_{1}&x_{2}\\
x_{3}&x_{4}
\end{pmatrix}
\begin{pmatrix}
\lambda&1\\
0&\lambda
\end{pmatrix} \\
 &= \begin{pmatrix}
 \lambda^2 x_{1}+ \lambda x_{3}& \lambda x_{1}+x_{3}+\lambda^2 x_{2}+ \lambda x_{4}\\
 \lambda^2 x_{3}& \lambda x_{3}+\lambda^2 x_{4} 
\end{pmatrix} \\
 XAX &= \begin{pmatrix}
x_{1}&x_{2}\\
x_{3}&x_{4}
\end{pmatrix}
 \begin{pmatrix}
\lambda&1\\
0&\lambda
\end{pmatrix}
\begin{pmatrix}
x_{1}&x_{2}\\
x_{3}&x_{4}
\end{pmatrix} \\
 &= \begin{pmatrix}
\lambda x_{1}^2 + x_{1}x_{3}+ \lambda x_{3}x_{2} &\lambda x_{1}x_{2} + x_{1}x_{4}+ \lambda x_{4}x_{2}\\
\lambda x_{1}x_{3} +x_{3}^2+ \lambda x_{3}x_{4}& \lambda x_{2}x_{3} + x_{4}x_{3}+ \lambda x_{4}^2
\end{pmatrix} 
\end{aligned}
\end{equation}
Upon equating and solving, we have, $x_{3}= 0$ or $-\lambda^2$. Now if $x_{3}=0$, as a consequence from above equations, we get $x_{1} = x_{4}= 0$ or $\lambda$. But, $x_{1},x_{4}= 0$ is not possible, since by Theorem \ref{main_thm_2}, when $\lambda \neq 0$, non-trivial solution $X$ is similar to A, hence trace of $A$ and $X$ must be equal, which is $2\lambda$. As a result, we have $x_{1},x_{4}=\lambda$, also which leads to $x_{2}=1.$ \\
Now, if $x_{3}= -\lambda^2$, then we have $x_{1},x_{4}= \lambda\pm\lambda \sqrt{x_2}$. But again, since $A$ and $X$ are similar, their trace must be equal. As a result,  $x_{1}$ should be conjugate of $x_{4}$, and $x_{2}$ has a free choice.

Note that, when the dimension of the coefficient matrix increases, there will be more free variable as $a$ in this case, which hardly follows any pattern w.r.t the dimension.
 \end{proof}
 
\begin{Example}\label{ex2} 
Let A be the Jordan block of size 2, with eigenvalue $\lambda = 0$, we have the following.
\begin{equation}
    \begin{aligned}
    AXA &= \begin{pmatrix}
0&1\\
0&0
\end{pmatrix}
\begin{pmatrix}
x_{1}&x_{2}\\
x_{3}&x_{4}
\end{pmatrix}
\begin{pmatrix}
0&1\\
0&0
\end{pmatrix} 
= \begin{pmatrix}
0&x_{3}\\
0&0
\end{pmatrix} \\
 XAX &= \begin{pmatrix}
x_{1}&x_{2}\\
x_{3}&x_{4}
\end{pmatrix}
\begin{pmatrix}
0&1\\
0&0
\end{pmatrix}
\begin{pmatrix}
x_{1}&x_{2}\\
x_{3}&x_{4}
\end{pmatrix}
= \begin{pmatrix}
x_{1}x_{3}&x_{1}x_{4}\\
x_{3}^2&x_{4}x_{3}
\end{pmatrix} \\
 \end{aligned}
\end{equation}
This gives $x_{3}=0, x_{1}x_{4}=0$ and $x_{2}$ has free choice over $K$. \\
Hence, $Sol_A = \{\begin{pmatrix}
a&\alpha\\
0&b
\end{pmatrix} |\  a,b,\alpha \in K,\ ab=0 \}$.
\end{Example}
\subsubsection{Some General Solutions when $\lambda = 0$}
\begin{Example}\label{ex3}
Let A be the $3\times 3$ Jordan block with eigenvalue 0, and the variable matrix be $X= \begin{pmatrix}
a & b & c\\ 
d & e & f\\
g & h & i
\end{pmatrix}$.
 Upon equating $AXA=XAX$, we have the following polynomials, whose zeros give $Sol_A$. \\
\begin{equation}
\begin{aligned}
 f_1 = ad+bg     &&  f_2 = ae+bh-d && f_3 = af+bi-e  \\
 f_4 = d^2+eg    && f_5 = de+eh-g && f_6 = df+ei-h\\
 f_7 = gd+hg && f_8 = ge+h^2 && f_9 =  gf+hi
\end{aligned}
\end{equation}
Let us find a Gr$\ddot{o}$bner basis for the ideal $\langle f_1,f_2,\dots,f_9 \rangle$ in the ring $\mathbb{C}[a,b,c,d,e,f,g,h,i]$, with respect to the lexicographic order for $a>b>\dots >i$. On simplifying from the Gr$\ddot{o}$bner basis, we get $g,h,d,e=0$ and $af+bi=0$, and other variables have free choice. Hence,
\[ Sol_A = \{ \begin{pmatrix}
a & b & c\\
0 & 0 & f\\
0 & 0 & i
\end{pmatrix} |\   af+bi = 0 \}.
\]
\end{Example}

\begin{Example}\label{examplenilpotent}
For the coefficient matrix A being Jordan block of size n with eigenvalue 0, one can generalise the solutions in Example \ref{ex3}, to the following solutions class. But note that, for $n >3$, this does not give the whole set of solutions.  \[ X= \begin{pmatrix}
0&a_{1}&a_{2}&....&a_{n-2}&\alpha\\
0&0&0&\dots & \sum_{i=1}^{n-3} a_{i} b_{i+1} &b_{1}\\
0&0&0&\dots&0&b_{2}\\
.&.&.&\dots&.&.\\
.&.&.&\dots&.&.\\
.&.&.&\dots&.&.\\
0&0&0&\dots&0&b_{n-2}\\
0&0&0&\dots&0&0
\end{pmatrix} \]

We can directly verify the above matrix belongs $Sol_A$.
\end{Example}
\begin{lemma}
    Let $A \in M_n(K)$ be in its Jordan-canonical form, then for a matrix, $M \in M_n(K)$, $AM = MA$ implies $M \in K[A]$.
\end{lemma}
\begin{proof}
    
\end{proof}

\begin{theorem}
Let A be the Jordan block of size $n+1$, $n\geq 3$, with eigenvalue 0. Any commuting solution to the YBE, with the coefficient matrix $A$, will have the form,\\

\[ X= \begin{pmatrix}
0&1&0&0&\dots&\alpha&\beta\\
0&0&1&0&\dots&0&\alpha\\
0&0&0&1&\dots&0&0\\
.&.&.&.&\dots&.&.\\
.&.&.&.&\dots&.&.\\
.&.&.&.&\dots&.&.\\
0&0&0&0&\dots&0&1\\
0&0&0&0&\dots&0&0
\end{pmatrix} \text{ or }
\begin{pmatrix}
    0&0&\dots&\alpha&\beta\\
    0&0&\dots&0&\alpha\\
    0&0&\dots&0&0\\
    .&.&\dots&.&.\\
    .&.&\dots&.&.\\
     .&.&\dots&.&.\\
     0&0&\dots&0&0\\
     0&0&\dots&0&0
\end{pmatrix}
\text{ where } \alpha, \beta \in K \]
\end{theorem}

\begin{proof}
Here we have $A= B$, as $\lambda =0$, hence if $X$ is a commuting solution to the YB equation, then $X$ must be a polynomial in $B$. Let, $X = a_{0}I+a_{1}B+...+a_{n-1}B^{n-1}+a_{n}B^n$  (note  $B^{n+1}=0$ ). 
Then,  
\begin{equation}
    \begin{aligned}
 BXB &= B(a_{0}I+a_{1}B+\dots+a_{n-1}B^{n-1}+a_{n}B^n) B \\
 &= a_{0}B^2+a_{1}B^3+\dots+a_{n-2}B^{n} \\
  &= \sum_{i= 2}^n a_{i-2}B^i \\
XBX &= (a_{0}I+a_{1}B+\dots+a_{n}B^n) B (a_{0}I+a_{1}B+\dots+a_{n}B^n) \\
 &= (a_{n-1}a_{0}+a_{n-2}a_{1}+\dots+a_{0}a_{n-1})B^n +\dots+(a_{1}a_{0}+a_{0}a_{1})B^2 + a_{0}^2B \\
 &= \sum_{i=1}^n \sum_{j+k = i-1}a_ja_k B^i
\end{aligned}
\end{equation}
Upon equating, we get 
\begin{equation}
    \begin{aligned}
 a_{0}&=0,\  2a_0a_1 = a_0 \\
a_{1}&= 2a_{2}a_{0}+a_{1}^2,(\implies a_{1}= 0 \text{ or } 1)\\
a_{2} &= a_{3}a_{0}+a_{2}a_{1}+a_{1}a_{2}+a_{0}a_{3}, (\implies a_{2}= 0)\\
a_{3} &= a_{4}a_{0}+a_{3}a_{1}+a_{2}^2+a_{1}a_{3}+a_{0}a_{4}, (\implies a_{3}=0)
 \end{aligned}
\end{equation}
Continuing this way, we obtain, in general,
\[ a_{n-2}=a_{n-1}a_{0}+a_{n-2}a_{1}+...+a_{1}a_{n-2}+a_{0}a_{n-1}, (\implies a_{n-2}=0)\]
and $a_{n-1}, a_{n}$ has free choice over $K$. Hence $X$ has a general form as stated above when $a_{1}=1$, and $a_{1}=0$ respectively.
\end{proof}

\section{Two Jordan blocks}
Let the coefficient matrix $\mathcal{A}$ has two Jordan blocks in its Jordan-canonical form, say $\begin{pmatrix}
A_1 & 0\\
0 & A_2
\end{pmatrix}$ where $A_i = \lambda_i I_{n_i} + B_i$. Rewriting the YB equation in block form with $X= \begin{pmatrix}
X_1 & X_2 \\
X_3 & X_4
\end{pmatrix}$, and expanding, gives us the set of four equations, namely:
\begin{equation}\label{two_jordan_block}
\begin{aligned}
    &X_1A_1X_1+X_2A_2X_3=A_1X_1A_1\\
    &X_1A_1X_2+X_2A_2X_4=A_1X_2A_2\\
    &X_3A_1X_1+X_4A_2X_3=A_2X_3A_1\\
    &X_3A_1X_2+X_4A_2X_4=A_2X_4A_2\\
\end{aligned}
\end{equation} 
 Also, note that when the coefficient matrix is invertible, any invertible solution is similar to the coefficient matrix, as we stated in  Lemma $\ref{both_invertible}$.

\begin{lemma}
If the coefficient matrix $\mathcal{A}$ has two Jordan blocks, $A_i$ with $\lambda_i \neq 0$, for $i\ = 1,2$, with the cyclic subspaces of $A_i$ being $P_i$ respectively. Then, for $X \in Sol_{\mathcal{A}}$,  $Ker(X) = P_1$ or $P_2$ or $P_1\bigoplus P_2$. 

\end{lemma}
\begin{proof}
 By Lemma \ref{preserver}, if $X \in Sol_{\mathcal{A}}$, $Ker(X)$ is an invariant subspace of $\mathcal{A}$. Hence, $Ker(X)$ should be one of $P_1,P_2, P_1 \bigoplus P_2, E_{\lambda_1} , E_{\lambda_2}$. Now, in light of the Corollary \ref{fact_one}, $E_{\lambda_1}, E_{\lambda_2}$ can not be the kernel, which leads to the result.

\end{proof}
\begin{lemma}\label{lemma_twoblock}
Let $A_1$ and $A_2$ be two Jordan blocks with eigenvalues $\lambda_1$ and $\lambda_2$, respectively, and $X_2 \in Sol_{A_2}$. Then, $A_1X_1A_2= X_1A_2X_2$ if and only if either $\sigma(A_1) \cap \sigma(X_2) \neq \phi $ or $X_1A_2 = 0$. Further, if $A_1, A_2$ are Jordan blocks with non-zero eigenvalues and of the same dimension, then either $A_1 = A_2$ or $X_1 = 0$. Also, when $A_1=A_2=A$, $X_1$ will have the form $ZS^{-1}A^{-1}$, where $Z \in K[A]$, and for some $S \in GL_n(K)$. 
\end{lemma}
\begin{proof}
Let $Y= X_1A_2$, then the equation $A_1X_1A_2= X_1A_2X_2$ becomes, $A_1Y=YX_2$. Then by Sylvester's theorem \ref{sylvester}, Y has a non-zero solution if and only if $\sigma(A_1) \cap \sigma(X_2) \neq \varnothing $.
Therefore, if $X_1A_2 \neq 0$, then $\sigma(A_1) \cap \sigma(X_2) \neq \varnothing $, or if $X_1A_2 = 0$, then $\sigma(A_1) \cap \sigma(X_2) = \varnothing $.\\
Let $X_1A_2 \neq 0$, and $A_1, A_2$ are Jordan blocks of the same dimension, with non-zero eigenvalues. As $X_2 \in Sol_{A_2}$, if $X_2$ is non-trivial, then $X_2 \simeq A_2$, i.e, $\sigma(X_2) = \{\lambda_2\} = \sigma(A_2)$. Also, we have $\sigma(A_1) \cap \sigma(X_2) \neq \varnothing$, i.e. $\sigma(X_2) = \{\lambda_1\} = \sigma(A_1)$, leads to $\lambda_1 = \lambda_2$. Since $A_1$ and $A_2$ are Jordan blocks and of the same dimension, we have, $A_1 = A_2$.\\
 If $X_1A_2 = 0$, as $A_2$ is invertible, we have $X_1 = 0$.\\
Now, when $A_1=A_2=A$, the equation $A_1X_1A_2= X_1A_2X_2$ implies $AX_1A= X_1AX_2$. By taking $Y=X_1A$, and since $X_2 \simeq A$, we have $X_2=SAS^{-1}$ for some S, and $AY-YSAS^{-1}=0$
$\Longrightarrow AYS-YSA=0$. Calling $YS$ as $Z$, we have $AZ=ZA$.
Since A is Jordan block, $Z \in K[A]$.

\end{proof}

\begin{theorem}\label{two_jordan_main}
Let $\mathcal{A}$ be a matrix having only two eigenvalues $\lambda_1, \lambda_2 \neq 0$, with the same algebraic multiplicity, and geometric multiplicity 1. Further, let us assume 0 is an eigenvalue of $X$. Let $U \in GL_n(K)$ be defined as, the columns of $U$ are generalized eigenvectors of $\mathcal{A}$,  such that $ U\mathcal{A}U^{-1}$ is Jordan-canonical form  of $\mathcal{A}$. Then one of the following is true. 
\begin{enumerate}
    \item $X=0$
    \item $X = U^{-1} \begin{pmatrix}
    0 & ZS^{-1}A \\
    0 & Y_2
    \end{pmatrix}U$  or $U^{-1} \begin{pmatrix}
    0 & 0 \\
    0 & Y_2
    \end{pmatrix} U$, where $Y_2 \in Sol_{A}$, $Z \in K[A]$, $A$ is the Jordan block corresponding to eigenvalue $\lambda= \lambda_1 = \lambda_2$, and for some $S \in GL_n(K)$
    \item $X = U^{-1} \begin{pmatrix}
    Y_1 & 0 \\
    ZS^{-1}A & 0
    \end{pmatrix} U$ or $U^{-1} \begin{pmatrix}
    Y_1 & 0 \\
   0 & 0
    \end{pmatrix} U$, where $Y_1 \in Sol_{A}$, $Z \in K[A]$, $A$ is the Jordan block corresponding to eigenvalue $\lambda= \lambda_1 = \lambda_2$, and for some $S \in GL_n(K)$
\end{enumerate}
\end{theorem}
\begin{proof}
Let us discuss the three cases for the kernel of $X$. If the kernel of $X$ is $P_1 \bigoplus P_2$ then clearly $X=0$. If the kernel is $P_1$, let us write the block form of $X$ in the basis of $P_1 \bigoplus P_2$. We get the following block form $X \simeq \begin{pmatrix}
0 & Y_1\\
0 & Y_2
\end{pmatrix}$ for some matrices $Y_1,Y_2$. i.e., For some $U \in GL_n(K)$, we have,
\[ Y = UXU^{-1} = U \begin{pmatrix}
   X_1 & X_2\\
X_3 & X_4 
\end{pmatrix} U^{-1} = \begin{pmatrix}
   0 & Y_1\\
0 & Y_2 
\end{pmatrix}, \]
where the column vectors of $U$ are from the basis of $P_1 \bigoplus P_2$, such that $ U\mathcal{A}U^{-1} =\mathcal{A'}$ is the Jordan-canonical form of $\mathcal{A}$.  
Then, rewriting the YBE as $\mathcal{A'}Y\mathcal{A'}= Y\mathcal{A'}Y$, where $\mathcal{A'}= U\mathcal{A}U^{-1}$, gives the original solution $X$, as $X = U^{-1}YU$. 

Let $\mathcal{A'}= \begin{pmatrix}
A_1 & 0\\
0 & A_2
\end{pmatrix}$ , where $A_i = \lambda_i I+ B_i$, for $i =1,2$, is the Jordan block corresponding to eigenvalue $\lambda_i$.
Let us write the YB equation, 
\[Y_2A_2Y_2=A_2Y_2A_2\]
\[A_1Y_1A_2=A_2Y_2A_2\]

The first equation is just an instance of YBE for the single Jordan block case, which by Theorem \ref{main_thm_2} and the fact that $det(A_2) \neq 0$, we get $X_2 \simeq A_2$ or 0. By Lemma \ref{lemma_twoblock}, we have, either $X_1 = 0$ or $X_1 = ZS^{-1}A^{-1}$, where $Z \in K[A]$. This gives us the solution stated in the second case.

The proof of the third case is exactly similar to the second case.
\end{proof}
\begin{Example} 
The following example will give insight into the significance of the Theorem \ref{two_jordan_main}, which otherwise would have been more complicated to solve.
Let $\mathcal{A} = \begin{pmatrix}
A&0\\
0&A
\end{pmatrix}$, where $A = \begin{pmatrix}
\lambda & 1\\
0 & \lambda
\end{pmatrix}$, for some $\lambda \neq 0$. Then by \ref{two_jordan_main}, one of the forms for the solution $X \in Sol_{\mathcal{A}}$ is $\begin{pmatrix}
X_1 &0\\
X_2 &0
\end{pmatrix}$, where $X_1 \in Sol_A$, and let $X_2$ be the variable matrix $\begin{pmatrix}
b&c\\
d&e
\end{pmatrix}$. Also, for $X_1$, let's choose one of the forms given in the special solutions Example \ref{ex1}, say  $X_1 = \begin{pmatrix}
\lambda+\lambda\sqrt{a}& a\\
-\lambda^2 & \lambda-\lambda \sqrt{a}
\end{pmatrix} $ for some $a \in K$. Then on simplifying the equation $\mathcal{A}X\mathcal{A}=X\mathcal{A}X$, it reduces to the following equations: \\
\begin{equation}
\begin{aligned}
   \text{(1) }\lambda^2(\sqrt{a}b-c \lambda)= b\lambda^2+ d \lambda &&
    \text{(2) } \lambda^2(\sqrt{a}d-c \lambda) = d \lambda^2 \\
    \text{(3) } ad\lambda - \sqrt{a}\lambda (d+e\lambda)=0  &&
    \text{(4) } d+e\lambda =  ab\lambda - \sqrt{a}\lambda (b+c\lambda)
 \end{aligned}
\end{equation}
Then for different cases of $\lambda, a$, we have the following solutions for $X_2$, which along with the given $X_1$, completes the solution $X$. 
\begin{itemize}
    \item[i]When $a=0$, $X_2 = \begin{pmatrix}
    b&\frac{e-b}{\lambda}\\
    -e\lambda & e
    \end{pmatrix} $ where $b,e \in K$.
    \item[ii] When $a \neq 0, a \neq 1, \lambda =1 $, 
    $X_2 = \begin{pmatrix}
      \frac{c}{\sqrt{a}-1}+\frac{e}{(\sqrt{a}-1)^2} & c\\
      \frac{e}{\sqrt{a}-1} & e
    \end{pmatrix}$ where $c,e \in K$.
    \item[iii]  When $a \neq 0, a \neq 1, \lambda \neq 1 $, 
    $X_2 = \begin{pmatrix}
    \frac{e}{(\sqrt{a}-1)^2} & 0\\
    \frac{e}{\sqrt{a}-1} & e
    \end{pmatrix}$ where $ e \in K$.
    \item[iv] When $a =1, \lambda \neq 1$, 
    $X_2 = \begin{pmatrix}
     b & 0\\
    0 &  0
    \end{pmatrix}$
    \item[v] When $a=1, \lambda = 1$, 
    $X_2 = \begin{pmatrix}
    b& c\\
    -c&0
    \end{pmatrix}$ where $b,c \in K$.
\end{itemize}

\end{Example}
\section{Pencils of solutions}
\begin{lemma}\label{pencil1}
Let $X \in Sol_A$, then for a matrix $M$ with $MA=0=AM$, for any $\alpha \in K$, $X+\alpha M \in Sol_A$.
 \end{lemma}

\begin{proof}
we have
we have
$A(X+\alpha M)A = AXA + \alpha AMA = AXA$
\[(X+\alpha M)A(X+\alpha M) = XAX+ \alpha XAM +\alpha MAX + \alpha MAM = XAX \]
Note that in this case, $M$ itself forms a solution. 
\end{proof}
\begin{lemma}
Let $X \in Sol_A$, for any $\alpha, \beta \in K$, $\alpha X + \beta M \in Sol_A$, if $MA=0=AM$ and $AXA = 0$. 
\end{lemma}
\begin{proof}
If $X \in Sol_A$, then for $\alpha \neq 0,1$, $\alpha X$ is a solution if $AXA=0$ \cite{YBE2}. Then the rest follows directly from the above Lemma \ref{pencil1}.
\end{proof}
When A is non-singular, $MA=0=AM$ demands M to be 0. This discussion is interesting when A is singular. Let A be a singular matrix in its Jordan form. Say $A = \begin{pmatrix}
J_1 & 0\\
0 & J_2
\end{pmatrix}$ where $J_1 = \begin{pmatrix}
0&1&0\\
0&0&1\\
0&0&0
\end{pmatrix}$, $J_2 = \begin{pmatrix}
\lambda & 1\\
0 & \lambda
\end{pmatrix}$ for some $\lambda \neq 0 \in K$. Then we can see that $M$ has the form $\begin{pmatrix}
0&0&\alpha&0&0\\
0&0&0&0&0\\
0&0&0&0&0\\
0&0&0&0&0\\
0&0&0&0&0
\end{pmatrix}$, where $\alpha \in K$. When $J_2$ is also Jordan block of 0, that is $J_2 = \begin{pmatrix}
0 & 1\\
0 & 0
\end{pmatrix}$, M will have the following form. $M = \begin{pmatrix}
0&0&a_1&0&a_2\\
0&0&0&0&0\\
0&0&0&0&0\\
0&0&a_3&0&a_4\\
0&0&0&0&0
\end{pmatrix}$, where $a_i \in K$. We can observe that when there is a Jordan block of non-zero eigenvalue, the corresponding rows and columns in M will be annihilated. Also, corresponding blocks in M are solutions to respective YBEs $J_iXJ_i = XJ_iX$. 
\begin{theorem}
Let $X_0, X_1 \in Sol_A$, for any $\lambda \neq \in K$, $X_0+\lambda X_1 \in Sol_A$ if and only if $AX_1A=0=X_1AX_1$ and $X_0AX_1+X_1AX_0 =0$.
\end{theorem}

\begin{proof}
$(\Rightarrow)$
We have $X_0AX_0 = AX_0A, X_1AX_1=AX_1A$.   
Which implies $A(X_0+\lambda X_1)A = AX_0A+ \lambda AX_1A = AX_0A$, as $AX_1A=0$. Also, we have
$(X_0+\lambda X_1)A(X_0+\lambda X_1)  = X_0AX_0$.

$(\Leftarrow)$ For any $\lambda \neq \in K$, if $X_0+ \lambda X_1$ is a solution, then,
\begin{equation}
    \begin{aligned}
A(X_0+\lambda X_1)A &= AX_0A+ \lambda AX_1A\\
&= (X_0+\lambda X_1)A(X_0+\lambda X_1)\\
&= X_0AX_0+\lambda X_0AX_1+ \lambda X_1AX_0+ \lambda^2 X_1AX_1.\\
\end{aligned}
\end{equation}
Then $\lambda AX_1A= \lambda X_0AX_1+ \lambda X_1AX_0+ \lambda^2 X_1AX_1$.
Since it's true for every $\lambda$, it's true for $\lambda = 1$ also, which gives 
$X_0AX_1+X_1AX_0 =0$. This implies $\lambda AX_1A=0= \lambda^2 X_1AX_1$ which is true for any $\lambda$, gives $AX_1A=0=X_1AX_1$.
\end{proof}
\begin{Example}
When A is the $3 \times 3$ Jordan block, with eigenvalue zero, in Example \ref{ex3} of special solutions, we have seen that X must have the form $\begin{pmatrix}
a&b&c \\
0&0&f \\
0&0&i
\end{pmatrix}$, such that $af+bi = 0$. Then, by taking $X_j = \begin{pmatrix}
a_j&b_j&c_j \\
0&0&f_j \\
0&0&i_j
\end{pmatrix} \in Sol_A$, for $j=0,1$, we can easily verify the condition for $X_0+\lambda X_1 \in Sol_A$. Note that $AX_jA=0$. Now,\\ 
$X_1AX_0+X_0AX_1 =  \begin{pmatrix}
0&0&a_1f_0+a_0f_1+b_1i_0+b_0i_1\\
0&0&0\\
0&0&0
\end{pmatrix} = 0$ , if $\lambda (a_1f_0+a_0f_1+b_1i_0+b_0i_1 ) = 0$.
This is agreeing with the general form of solution when 
\[\begin{pmatrix}
a_0+\lambda a_1&b_0+\lambda b_1&c_0+\lambda c_1 \\
0&0&f_0+\lambda f_1 \\
0&0&i_0+\lambda i_1
\end{pmatrix} \in Sol_A.\]
\end{Example}
\section{Conclusion}
In this paper, solutions to the Yang-Baxter equation has studied for coefficient matrix A being single Jordan blocks and A having two Jordan block in the canonical form. A complete characterization of solutions is still an open question. Although the solutions are harder to find, it inspire to approach through many parallel ways including via Gr$\ddot{o}$bner basis. In future we would like to look into those kinds of solutions more, also these theories ignite light to finding solutions for similar interesting matrix equations such as $(i)\text{ }\Phi_A(X)X = 0$, $(ii)\text{ } \Phi_A(X)X = \Phi_X(A)A$.
\section{Bibliography}

\bibliographystyle{abbrv}
\bibliography{YB}
\end{document}